\documentclass[a4j,12pt]{amsart}
\usepackage{amsmath,amsthm,amssymb}
\usepackage{bigdelim,multirow}
\newtheorem{thm}{Theorem}
\newenvironment{thmbis}[1]
  {\renewcommand{\thethm}{\ref{#1}$'$}%
   \addtocounter{thm}{-1}%
   \begin{thm}}
  {\end{thm}}
\newtheorem{prop}{Proposition}

\newtheorem{cor}{Corollary}
\newtheorem{lem}{Lemma}
\newtheorem{rem}{Remark}

\begin{document}
\title{On the existence of local frames of CR vector bundles}
\author{Tomonori Kajisa}
\thanks{Tomonori Kajisa; ACRO YASHIRODAI 103, Yashirodai 3-59, Meitou-Ku, Nagoya, 465-0092, Japan, \ t-kajisa@citv.jp}
\date{}
\maketitle
\begin{abstract}
 Given a CR manifold $D$, we shall show that existence of a CR local frame of a certain CR vector bundle over $D$ is equivalent to the local imbeddability of $D$. This will imply that there exists a CR vector bundle which doesn't have CR local frames. Using this bundle, we shall construct CR line bundles over $3$-dimensional non-imbeddable CR manifolds which don't have CR local frames. 
\end{abstract}

\section{Introduction}
\ In CR geometry, CR vector bundle is a basic notion. In contrast to holomorphic vector bundles over complex manifolds, CR local frames do not always exist, although it was shown by Webster \cite{Web} that CR vector bundles always admit CR local frames if the manifold is strongly pseudoconvex (spc) and of dimension $\geq 7$.
 In this paper we shall say that a CR vector bundle is CR framable (framable for 
short) if it has CR local frames around any point and consider framability problem of a CR vector bundle over a $3$-dim CR manifold mainly. First we discuss a 
relation between local imbeddability of a CR manifold and framability of a CR vector bundle. This relation was studied in \cite{Jac} and \cite{Web}. We refine the result of Webster \cite{Web}. 
\begin{thm} \label{thm1}
 Let $(D,T^{0,1}D)$ be a $2n-1 \ (n\geq 2)$ dimensional CR manifold.
 Then $(D,T^{0,1}D)$ has a CR coframe locally if and only if it admits a local imbedding to $\mathbb{C}^{n}$.
\end{thm} 

A CR coframe is a CR frame of a certain CR vector bundle (see section \ref{sec2}). So it will imply that there exists a non-framable CR vector bundle over any non-imbeddable CR manifold. Particularly we obtain
 a non-framable CR vector bundle of rank $2$ over every non-imbeddable 3-dim CR manifold. (There exist a lot of examples of non-imbeddable spc manifolds. See \cite{JaTr}.)
 Next we ask whether there exist non-framable CR line bundles over a non-imbeddable 3-dim CR manifold. We introduce CR vector bundle structure space $\mathcal{H}''(F)$ over a $\mathbb{C}$-vector bundle $F$ and discuss this problem differential geometrically. Using the non-framable CR vector bundle of rank $2$ we can construct non-framable CR line bundles under some condition.
 
\begin{thm} \label{thm2}
 Let $D$ be a $3$-dim CR manifold, $F'$ be a $\mathbb{C}$-line bundle of over $D$, $F''$ be a $\mathbb{C}$-vector bundle of rank $2$ over $D$ and $p\in D$. Assume the existence of a non-framable CR vector bundle structure $\omega_0\in \mathcal{H}''(F'')$.  
If there are no CR local frames but there is a nowhere-vanishing CR local section around $p$ for $\omega_0$,  then there exist non-framable CR line bundle structures in $\mathcal{H}''(F')$. 
\end{thm}
Furthermore it is shown that if there exists a non-framable CR line bundle structure, we can find a lot of non-framable structures. 
\begin{thm} \label{thm3}
Let $D$ be a $3$-dim CR manifold, $E'=\mathbb{C}\times D$, and $F'$ be a $\mathbb{C}$-line bundle over $D$.
If there is a non-framable CR line bundle structure over $E'$, then there exist non-framable structures in $\mathcal{H}''(F')$ arbitrarily close to any framable structure in $\mathcal{H}''(F')$. 
\end{thm}

\section{Preliminaries}  \label{sec2}
Let $D$ be a $2n-1 \ (n\geq 2)$ dimensional $C^{\infty }$ manifold and $T^{0,1}D$ be a subbundle of $\mathbb{C}TD:=\mathbb{C}\otimes TD$ of rank of $n-1$ such that $T^{0,1}D\cap \overline{T^{0,1}D}=\{0\}$
and $[\Gamma(T^{0,1}D), \Gamma(T^{0,1}D)]\subset \Gamma(T^{0,1}D)$, where $\Gamma(T^{0,1}D)$ denotes the set of $C^\infty$ sections of $T^{0,1}D$ on $D$. The pair $(D,T^{0,1}D)$ is called a CR manifold. We set $T^{1,0}D=\overline{T^{0,1}D}$.
It is possible that we define a CR manifold in another way using differential forms. Namely, let $G$ be a subbundle of $\mathbb{C}TD^*$ of rank $n$ 
and $\mathcal{I}(G)$ be the exterior ideal of complex differential forms on $D$ generated by $G$. If $G+\overline{G}=\mathbb{C}TD^*$ and $d\mathcal{I}(G)\subset\mathcal{I}(G)$ are satisfied, the pair $(D,G)$ is called a CR manifold.
 In these two definitions ${T^{0,1}D}^\bot:=\{w\in \mathbb{C}TD^*; w(v)=0,  \ for \ \mathrm{any}\ v\in T^{0,1}D\} $ coincides with $G$. Let $\mathcal{A}^p (0\leq p)$ be the sheaf of $\mathbb{C}$-valued $p$-forms, let $\mathcal{A}^{p}(F)$ be the sheaf of $F$-valued $p$-forms for a $C^\infty$ $\mathbb{C}$-vector bundle $F$ 
 and let $\Gamma(\mathcal{A}^p)$ be sections of $\mathcal{A}^p$ over $D$.
 We may define locally free subsheaves of $\mathcal{A}^0$ modules 
 \[\hat{\mathcal{A}}^{p,q}=\{\omega\in \mathcal{A}^{p+q}| v_0\wedge v_1\wedge \cdots \wedge v_q \lrcorner \omega=0, \ \mathrm{for \ any} \ v_0,\cdots,v_q\in T^{0,1}D\}\] for $p \geq 1$ and $q \geq 0$, and set $\hat{\mathcal{A}}^{0,q}=\mathcal{A}^q\ (q\geq 0)$ and $\hat{\mathcal{A}}^{p,-1}=0$, where $\lrcorner$ denotes the interior product.
We now define smooth $(p,q)$-forms $\mathcal{A}^{p,q}$ by 
\[\mathcal{A}^{p,q}:=\hat{\mathcal{A}}^{p,q}/\hat{\mathcal{A}}^{p+1,q-1}\cong \mathcal{A}^0(\wedge^p(T^{0,1}D^\bot)\otimes \wedge^q T^{0,1}D^*).\]  
From the integrability condition $d\mathcal{I}(T^{0,1}D^\bot)\subset \mathcal{I}(T^{0,1}D^\bot)$, we have $d(\hat{\mathcal{A}}^{p,q})\subset \hat{\mathcal{A}}^{p,q+1}$. So we can define an operator $\overline{\partial}_b:\mathcal{A}^{p,q}\rightarrow \mathcal{A}^{p,q+1}$ as the exterior derivative $d$ composed with the projection $\pi: \hat{\mathcal{A}}^{p,q}\rightarrow \mathcal{A}^{p,q}$ as follows. 
\[\overline{\partial}_b [v]= \pi\cdot  dv \ \ \ \mathrm{for} \ [v]\in \mathcal{A}^{p,q}. \]

For each fixed $p$, $\{\mathcal{A}^{p,q}\}$ forms a complex. We will often call $\{\mathcal{A}^{0,q}\}$ 
simply the $\overline{\partial}_b$ complex. We refer the reader to \cite{Web} and \cite{Leb} to follow up the fundamental materials of CR manifolds. 
 Let $F$ be a $\mathbb{C}$-vector bundle of rank $r$ over $(D,T^{0,1}D)$. A CR vector bundle structure over $F$ is defined by a linear differential operator $\overline{\partial}_F: \mathcal{A}^0(F) \rightarrow \mathcal{A}^0(T^{0,1}D^*\bigotimes F)$ such that $\overline{\partial}_F(af)=(\overline{\partial}_ba)f+a\overline{\partial}_Ff$ for 
 $a \in \mathcal{A}^0,\  f\in \mathcal{A}^0(F)$ and $\overline{\partial}_{F} \cdot \overline{\partial}_{F}=0$ hold, where $\overline{\partial}_{F}$ is extended to $\overline{\partial}_{F}: \mathcal{A}^0(T^{0,1}D^*\bigotimes F) \rightarrow \mathcal{A}^0(\bigwedge ^2T^{0,1}D^*\bigotimes F)$ so that $\overline{\partial}_{F}\phi(X,Y)=\frac{1}{2}\{(\overline{\partial}_F\phi (Y))(X)-(\overline{\partial}_{F}\phi (X))(Y)-\phi([X,Y])\}$ holds for any $\phi\in \mathcal{A}^0(T^{0,1}D^*\bigotimes F)$ and any $X,Y\in \mathcal{A}^0(T^{0,1}D)$ and, $\overline{\partial}_{F} \cdot \overline{\partial}_{F}$ means their composition. The pair $(F,\overline{\partial}_F)$ is called a CR vector bundle. Let $e=<e_i>(1\leq i\leq r)$ be a local frame on an open set $U\subset D$ and
  let $\overline{\partial}_{F}e=\omega e$, where $\omega$ is a $\mathcal{A}^{0,1}$-valued $\textrm{r}\times\textrm{r}$ matrix function. Then $\omega$ satisfies $\overline{\partial}_b\omega-\omega\wedge \omega=0$ from the integrability condition $\overline{\partial}_{F} \cdot \overline{\partial}_{F}=0.$ 
Let $e'$ be another local frame on $U$. Then, there is a $\textrm{GL}(r,\mathbb{C})$ valued function $a$ such that $e'=ae$. Then $\overline{\partial}_{F}e'=(\overline{\partial}_ba)e+a\overline{\partial}_Fe=(\overline{\partial}_ba+a\omega)e$. If there exists a local section $u$ of $F$ such that $\overline{\partial}_{F}u=0$, we call it a CR local section and 
 if a set of nowhere-vanishing CR sections forms a local frame of $F$, we call it a CR local frame.  
A CR vector bundle has a CR local frame around $p \in D$ if and only if a non-linear PDE $a^{-1}\overline{\partial}_ba=-\omega$ has a local solution such that $\det a\neq 0$ around $p$. We say that a CR vector bundle is CR framable, or framable for short if there exist CR local frames everywhere. Examples of CR vector bundles are given in \cite{Leb}. CR vector bundles $\bigwedge^p ({T^{0,1}D})^\bot \ (1\leq p \leq n)$ are particularly important.
Because they are determined by CR structure of the base space $D$. A section $\phi$ of $\bigwedge^p ({T^{0,1}D})^\bot$ such that $d\phi \in \bigwedge^{p+1} (T^{0,1}D)^\bot$ is called a 
CR p-form. A frame of $T^{0,1}D^\bot$ composed of CR 1-forms is called a CR coframe. A CR $n$-form is also important. If a Levi non-degenerate CR manifold has a nowhere-vanishing CR $n$-form, it admits a pseudo-Einstein structure. (See \cite{Lee}).

The following formula for a $\mathbb{C}$-line bundle over a CR manifold is easily proved.
\begin{prop} Let $(D,T^{0,1}D)$ be a CR manifold and $f,g$ be nowhere-vanishing
 functions on an open set $U\subset D$. Then

\begin{equation}
(fg)^{-1}\overline{\partial}_b(fg)=f^{-1}\overline{\partial}_bf+g^{-1}\overline{\partial}_bg.
\end{equation}
\end{prop}

\section{Local imbeddability and framability of a CR vector bundle}
 Let $(D,T^{0,1}D)$ be a $2n-1 \ (n\geq 2)$ dimensional CR manifold and $p \in D$. The CR imbeddimg problem can be described from the viewpoint related to local $1$-parameter group of
 CR diffeomorphism. We shall quote several lemmas from \cite{Jac}.
For a real vector field $X$ let $\mathcal{L}_X\omega$ denote the Lie derivative acting on forms and vector fields.  If 
$Y=X_1+iX_2$ is a complex vector field, $\mathcal{L}_Y$ means the operator $\mathcal{L}_{X_1}+i\mathcal{L}_{X_2}$.
Note that the identity
\begin{equation}
\mathcal{L}_Y\omega=d(i_Y\omega)+i_Y(d\omega)
\end{equation}
is valid, where $\omega$ is any differential form.
\begin{lem}
 The following are equivalent:
\begin{enumerate}
\item $(D,T^{0,1}D)$ is locally imbeddable around $p$. 
 \item There exists a vector field $Y$ around $p$ with $\mathcal{L}_YT^{0,1}D\subset T^{0,1}D$ and $Y_p \notin T^{0,1}_p D+T^{1,0}_p D$.
\end{enumerate}
\end{lem}

\textit{Proof}.\ See \cite{Jac}.

\begin{lem}
For any vector field $Y$ the following are equivalent.
\begin{enumerate}
\item $\mathcal{L}_YT^{0,1}D\subset T^{0,1}D$
\item $\mathcal{L}_Y\bigwedge^{n} (T^{0,1}D)^\bot \subset \bigwedge^{n} (T^{0,1}D)^\bot$.
\item For every nowhere-vanishing section $\Omega$ of $\bigwedge^{n} (T^{0,1}D)^\bot $ there is some function $\lambda$ 
such that $\mathcal{L}_Y\Omega=\lambda \Omega$.
\item There is some nowhere-vanishing section $\Omega$ of $\bigwedge^{n} (T^{0,1}D)^\bot$ and some function $\lambda$ such that $\mathcal{L}_Y\Omega=\lambda \Omega$.
\end{enumerate}
\end{lem}
\textit{Proof}.\ See \cite{Jac}.

\begin{proof}[Proof of Theorem \ref{thm1}]
Let $U$ be an open set in $D$ such that the local triviality (in the sense of $C^{\infty}$) $TD|_U\cong U\times \mathbb{R}^{2n-1}$ 
holds.
Then, there is a nowhere-vanishing real vector field $T$ on $U$ and we have a decomposition
\begin{equation}
\mathbb{C}TU=T^{0,1}U+T^{1,0}U+\mathbb{C}T.
\end{equation} 
From the canonical isomorphisms $T^{0,1}U^{*}\cong (T^{1,0}U+\mathbb{C}T)^\bot, \mathbb{C}T^{*}\cong (T^{0,1}U+T^{1,0}U)^\bot$, 
\begin{equation}
\mathbb{C}TU^{*}=T^{0,1}U^{*}+T^{1,0}U^{*}+\mathbb{C}T^{*}
\end{equation}also holds.

Let $\Gamma(T^{1,0}U)=<v_i>_{1\leq i\leq n-1}$. Then $\Gamma(\mathbb{C}TU)=<v_i,\overline{v_i},T>_{1\leq i\leq n-1}$.
Taking dual basis, $\Gamma(\mathbb{C}TU^{*})=<u_i,\overline{u_i},\eta>_{1\leq i\leq n-1}$, where $u_i(v_i)=1$ and $\eta(T)=1$.

Assume the existence of a CR coframe $<\theta_i>_{1\leq i\leq n}$ on $U$. Set $\Omega=\theta_1\wedge \cdots \wedge\theta_n$. 
We want to find $Y\notin \Gamma(T^{0,1}U+T^{1,0}U)$ such that $\mathcal{L}_Y\Omega=\lambda \Omega$ for some function $\lambda$.
$\mathcal{L}_Y\Omega=d(i_Y\Omega)+i_Y(d\Omega)=d(\sum _{i=1}^{i=n}(-1)^{i+1}\theta_i(Y)\theta_1\wedge \cdots \hat{\theta_i}\wedge \cdots \wedge\theta_n)$.
Set $Y=\sum_{i=1}^{i=n-1}f_iv_i+f_nT$ for some functions $f_i \ (1\leq i\leq n)$. We determine $f_i$ so that $Y$ satisfies the condition
above. $<\theta_i>$ can be written as follows.
\begin{equation}
\left(
\begin{array}{ccc}
\theta_1, & \cdots & ,\theta_n
\end{array}
\right)=
\left(
\begin{array}{cccc}
u_1, & \cdots & ,u_{n-1} & ,\eta
\end{array}
\right)A    \label{eqm1}
\end{equation} for some $A$ such that $\det A\neq 0$.
From (\ref{eqm1}),
\begin{equation}
\left(
\begin{array}{ccc}
\theta_1(Y),  &
 \cdots &
 ,\theta_n(Y)
\end{array}
\right)=
\left(
\begin{array}{ccc}
f_1,  &
 \cdots &
 ,f_{n} 
\end{array}
\right)A.
\end{equation}
For $i=1, \cdots, n$, set
\begin{equation}
 \left(
\begin{array}{ccc}
f_1,  &
 \cdots &
 ,f_{n} 
\end{array}
\right)=(0,\cdots,0,\overset{i}{\check{1}},0,\cdots,0)A^{-1}.
\end{equation}Then, from $\textrm{rank}_{\mathbb{C}}A=n$ we can obtain $\left(\begin{array}{ccc}
f_1,  &
 \cdots &
 ,f_{n} 
\end{array}\right)$ such that $f_n\neq 0$ for some $i\ (1\leq i\leq n)$. Then $\mathcal{L}_Y\Omega=\lambda \Omega$ holds for $Y=\sum_{i=1}^{i=n-1}f_iv_i+f_n T$ and CR imbeddability is shown.  
The converse is trivial. Let $\iota$ be a CR imbedding map from $U$ to $\mathbb{C}^n$ and $(z_1,\cdots,z_n)$ be a coordinate in $\mathbb{C}^n$. Then $<\iota^{*}dz_i> \ (1\leq i\leq n)$ is a CR coframe on $U$.

\end{proof}
\begin{rem}
As examples of non-imbeddable CR manifolds besides $3$-dim spc manifolds, a class of CR twister manifolds is also famous.
It was given by LeBrun \cite{Leb}.
\end{rem}
\section{Non-framable CR vector bundle structures}
\ Let $D$ be a $3$-dim CR manifold, $E= D \times \mathbb{C}^r$ and $p\in D$. $E$ has a trivial CR vector bundle structure $\overline{\partial}_b$, so we can write any CR vector bundle structure over $E$ as $\overline{\partial}_{E}=\overline{\partial}_b + \omega$,
where $\omega \in \Gamma(\mathcal{A}^{0,1}(\textrm{End} E))$. We regard $\mathcal{A}^{0,1}(\textrm{End} E)$ as the $\mathcal{A}^{0,1}$-valued $\textrm{r}\times\textrm{r}$ matrix space $\mathcal{M}(\mathrm{r},\mathcal{A}^{0,1})$.
As the integrability condition of $\overline{\partial}_{E}$, $\omega$ satisfies $\overline{\partial}_b\omega + \omega\wedge \omega=0$. Since $\mathcal{A}^{0,2}=0$, we have $\mathcal{H}''(E)=\{\overline{\partial}_b + \omega; \omega\in \Gamma(\mathcal{M}(\mathrm{r},\mathcal{A}^{0,1})), \overline{\partial}_b\omega + \omega\wedge \omega=0
 \}=\{\overline{\partial}_b + \omega; \omega \in \Gamma(\mathcal{M}(\mathrm{r},\mathcal{A}^{0,1}))\}$, where $\mathcal{H}''(E)$ denotes the set of all CR vector bundle structures over $E$. We choose the natural frame $e=<e_i>_{1\leq i \leq r}$ of the trivial bundle $E$. Then for $\overline{\partial}_{E}=\overline{\partial}_b+\omega$, 
 $\overline{\partial}_{E}e = \omega e$ holds. In this section, we shall ask whether there exist non-framable CR line bundles over a $3$-dim non-imbeddable CR manifold $D$. In \cite{Hor}, H\"{o}rmander gives a necessary condition for a linear PDE $Pu=f$ to have a local solution for every $\mathbb{C}$-valued function $f\in C^\infty$. (See Theorem 6.1.1, Theorem 6.1.2 in \cite{Hor}.) Since the PDE for framability of a CR line bundle is $a^{-1}\overline{\partial}_ba=-\omega$ and it is reduced to a $\overline{\partial}_b$-equation $\overline{\partial}_b(\log a)=-\omega$. However it is hard to check whether H\"{o}rmander's condition holds or not. So we will try another approach using a non-framable CR vector bundle structure $\overline{\partial}_{{T^{0,1}D}^\bot}$ obtained in the previous section. As a result we give an answer partially. In the end we shall ask how many non-framable CR line bundle structures exist in a CR line bundle structure space and how they exist there. Through this section, note the following two facts. For any $\mathbb{C}$-vector bundle $F$ over a $3$-dim CR manifold $D$, there exist CR vector bundle structures (i.e. $\mathcal{H}''(F)\neq\emptyset$). This is verified in the same way as construction of connections in vector bundles. (Take a covering $\{U_\alpha\}$ of $D$ and the associate partition of unity $\{\rho_\alpha\}$. Then set $\omega_\alpha=\sum_\beta \rho_\beta\cdot(\overline{\partial}_b A^{-1}_{\alpha \beta}\cdot A_{\alpha \beta})$, where $\{A_{\alpha \beta}\}$ is a family of transition functions.) CR vector bundle can be defined through a connection satisfying a certain integrability condition (see \cite{Web}), so checking that on a $3$-dim CR manifold $D$ any connection satisfies this integrability condition is another way. The key is that rank $T^{0,1}D^\ast=1$ $(\mathcal{A}^{0,2}=0).$ The second fact is that $E$ is framable around $p$ for any CR vector bundle structure over $E$ if and only if a $\mathbb{C}$-vector bundle $F$ over $D$ of rank $r$ is framable around $p$ for any CR vector bundle structure over $F$. This is also easily verified from $\mathcal{H}''(F)\neq\emptyset$ and $\mathcal{A}^{0,2}=0$. By these facts we may prove Theorem \ref{thm2} and Theorem \ref{thm3} in the following simpler forms. 

\begin{thmbis}{thm2}
 Let $D$ be a $3$-dim CR manifold, $E'= D \times \mathbb{C}$, and $E''= D \times \mathbb{C}^2$. Assume the existence of a non-framable CR vector bundle structure $\omega_0\in \mathcal{H}''(E'')$. If there are no CR local frames but there is a nowhere-vanishing CR local section around $p$ for $\omega_0$, then there exist non-framable CR line bundle structures in $\mathcal{H}''(E')$. 
\end{thmbis}
\begin{thmbis}{thm3}  
Let $D$ be a $3$-dim CR manifold and $E'= D \times \mathbb{C}$. If there is a non-framable CR line bundle structure over $E'$, then there exist non-framable structures arbitrarily close to any framable structure in $\mathcal{H}''(E')$.  
\end{thmbis}

 \begin{prop} \label{Pro4}
 Let $E'=D \times \mathbb{C}$ be a $\mathbb{C}$-line bundle over a $2n-1 \ (n\geq 2)$ dimensional CR manifold $(D,T^{0,1}D)$. Then all CR line bundle structures over $E'$ are framable if and only if $\underrightarrow{\lim}_{p\in U}H^{0,1}(U)=0$ for every $p\in D$, where 
 $U$ runs through the neighborhoods of $p$. 
\end{prop}
\begin{proof} In the $\mathbb{C}$-line bundle case, the set of all CR line bundle structures over $E'$ is $\{\overline{\partial}_b+\omega; \omega\in \Gamma(\mathcal{A}^{0,1}), \overline{\partial}_b\omega=0\}$.
And the PDE for framability can be written $a^{-1}\overline{\partial}_ba=\overline{\partial}_b(\log a)=-\omega$. 
Suppose a PDE $\overline{\partial}_bf=\omega$ can't be solved for some $\omega\in \Gamma(\mathcal{A}^{0,1})$ such that $\overline{\partial}_b\omega=0$, where $f$ is an unknown function. Then this $\omega$ gives a non-framable CR line bundle structure. If a PDE $\overline{\partial}_bf=-\omega$ can be solved locally around every point in $D$ for any $\omega\in \Gamma(\mathcal{A}^{0,1})$ such that $\overline{\partial}_b\omega=0$, $\overline{\partial}_b(\log a)=-\omega$ has a nowhere-vanishing local solution $a=e^f$. Therefore $\omega$ is framable.
\end{proof}
 \begin{prop} \label{Pro}
 Let $D$ be a $3$-dim CR manifold, $E'= D \times \mathbb{C}^r$ and $\omega\in\mathcal{H}''(E')$. Put $\omega'= \overline{\partial}_b S^{-1}S+S^{-1} \omega S$ for a $\textrm{GL}(r,\mathbb{C})$ valued function $S\in\Gamma(D\times \textrm{GL}(r,\mathbb{C}))$.
Then $\omega$ is framable if and only if $\omega'$ is framable.
 \end{prop}
 
 \begin{proof}
The PDE for framability of $\omega$ is 
 \begin{equation} 
 a^{-1}\overline{\partial}_ba=-\omega. \label{eq1}
 \end{equation}

 We consider the PDE
 \begin{equation}
 a'^{-1}\overline{\partial}_ba'=-\omega'= -(\overline{\partial}_bS^{-1}S+S^{-1}\omega S). \label{eq2}
\end{equation}
Noting that 
\begin{equation}
\overline{\partial}_b(S^{-1}S)=\overline{\partial}_bS^{-1}S+S^{-1}\overline{\partial}_bS=0, \label{eqI}
\end{equation}
(\ref{eq2}) can be written as
\begin{equation}
a'^{-1}\overline{\partial}_ba'-S^{-1}\overline{\partial}_bS=-S^{-1}\omega S. \label{eq2'}
\end{equation}

Here, we consider a PDE
\begin{equation}
 -\omega=(a'S^{-1})^{-1}\overline{\partial}_b(a'S^{-1})=S(a'^{-1}\overline{\partial}_ba'-S^{-1}\overline{\partial}_bS
)S^{-1}. \label{eq3}
\end{equation}
  (\ref{eq3}) is solvable if and only if (\ref{eq2'}) is solvable. This shows that $\omega$ is framable if and only if $\omega'$ is framable.
 
 \end{proof}

\begin{proof}[Proof of Theorem \ref{thm2}']
\ Let \ $\omega_0=\left(
\begin{array}{cc}
\omega_1^1 & \omega_1^2 \\
\omega_2^1 & \omega_2^2
\end{array}
\right)$ be some non-framable CR vector bundle structure in $\mathcal{H}''(E'')$
and
$S'=\left(
\begin{array}{cc}
s_1' & 0 \\
0 & s_2'
\end{array}
\right)$ be a $\textrm{GL}(2,\mathbb{C})$ valued matrix function. Let 
\begin{equation}
\omega_0'=\overline{\partial}_b S'^{-1}S'+S'^{-1} \omega_0 S'= 
\left(
\begin{array}{cc}
-s_1'^{-1}\overline{\partial}_bs_1'+\omega_1^1 & s_1'^{-1}s_2'\omega_1^2 \\
s_2'^{-1}s_1'\omega_2^1 & -s_2'^{-1}\overline{\partial}_bs_2'+\omega_2^2         
\end{array}
\right) \label{eq.n1} .  
\end{equation}
Then the PDE $-s_1'^{-1}\overline{\partial}_bs_1'+\omega_1^1=0$ or $-s_2'^{-1}\overline{\partial}_bs_2'+\omega_2^2=0$
can be solved if all CR line bundle structures are framable, where $s_1'$ and $s_2'$ are unknown functions. By picking up these solutions, we may assume 
 \begin{equation}\omega_0'=
\left(
\begin{array}{cc}
0 & s_1'^{-1}s_2'\omega_1^2 \\
s_2'^{-1}s_1'\omega_2^1 & 0
\end{array}
\right)   \label{eq.n2}
\end{equation} around $p\in D$.

If we assume that the second component of the local frame is a nowhere vanishing CR local section, we can set $\omega_0=\left(
\begin{array}{cc}
\omega_1^1& \omega_1^2 \\
0 & 0
\end{array}
\right)$. In addition, using the above argument we can reset $\omega_0'=
\left(
\begin{array}{cc}
0& \omega_1' \\
0 & 0
\end{array}
\right)$. 
If we set $S=\left(
\begin{array}{cc}
1 & s \\
0 & 1
\end{array}
\right) \ (s\neq 0)$, then 
$\overline{\partial}_b S^{-1}S+S^{-1} \omega_0'S $
\begin{equation}
=\left(
\begin{array}{cc}
0 & -\overline{\partial}_bs+\omega_1'
  \\
0 & 0

\end{array} 
\right) \label{eq.n3}.
\end{equation}

If all CR line bundle structures are framable, we can pick up $s(\neq 0)$ such that $-\overline{\partial}_bs+\omega_1'=0$, and
it is contradictory because $\omega=0$ is framable around $p$.
\end{proof}
Theorem \ref{thm2} implies that, on a $3$-dim non-imbeddable CR manifold $D$, if there is a nowhere-vanishing CR 1-form locally around every point in $D$, there exist non-framable CR line bundle structures over any $\mathbb{C}$-line bundle $F'$.

\begin{cor} \label{cor1}
Let $D$ be a $3$-dim non-imbeddable CR manifold and $F'$ be a $\mathbb{C}$-line bundle. If there is a local CR function $f$ such that $df_p\neq 0$ at every $p\in D$, there exist non-framable CR line bundle structures in $\mathcal{H}''(F')$.
\end{cor}
\begin{proof}
$df$ is a nowhere-vanishing CR 1-form.
\end{proof}
\begin{rem}
Corollary \ref{cor1} also follows from Theorem $2$ in \cite{Jac}.
\end{rem}
Hereafter, framability of CR vector bundle structures around a framable structure $\omega_1\in\mathcal{H}''(E)$ 
is discussed. We consider framability of a CR vector bundle structure $\omega_1+\omega_\delta$, which is
a perturbation of $\omega_1$. Since $\omega_1$ is framable, there exists a $\textrm{GL}(r,\mathbb{C})$ valued $a_1$ such that
\begin{equation}
a_1^{-1}\overline{\partial}_ba_1=-\omega_1.
\end{equation}
The PDE for framability of $\omega_1+\omega_\delta$ is 
\begin{equation}
a^{-1}\overline{\partial}_ba=-(\omega_1+\omega_\delta)=-(\overline{\partial}_b a_1^{-1}a_1+\omega_\delta). \label{eq4}
\end{equation}
Set $\omega_\delta= a_1^{-1}\omega_\delta' a_1$. Then from Proposition \ref{Pro}, $\omega_1+\omega_\delta$ is framable
if and only if $\omega_\delta'$ is framable. This implies that if we can construct arbitrarily small perturbations
 which are non-framable, we can find non-framable structures around every framable structure in $\mathcal{H}''(E)$.
From here we consider the case of CR line bundles. We set $\omega_\delta'=\delta\omega_0,0<\delta\leq 1$. Let $c=\inf\{\delta;\omega_\delta' \text{ is non-framable}\}$. If $c=0$, we can obtain the small perturbations as above. 
We consider the case $c>0$. In this case, $\omega_\delta'\ (0<\delta<c)$ are framable. 
For $\delta_1\ (0<\delta_1<c)$, there exists $L_{\delta_1}$ such that 
\begin{equation}
L_{\delta_1}\overline{\partial}_bL_{\delta_1}^{-1}=-\delta_1\omega_0\ 
\end{equation}
Let $\delta_2\geq c$ and $\omega_{\delta_2}={\delta_2}\omega_0$ be non-framable. Then,
\begin{equation}
\overline{\partial}_bL_{\delta_1}^{-1}L_{\delta_1}+L_{\delta_1}^{-1}(\delta_2\omega_0)L_{\delta_1}=(\delta_2-\delta_1)\omega_0.
\end{equation}
Therefore, from Proposition \ref{Pro} 
$(\delta_2-\delta_1)\omega_0$ are non-framable and $\delta_2-\delta_1>0$ can be arbitrarily small.
It's contradictory to $c>0$. The argument above proves Theorem \ref{thm3}'.

\section*{Acknowledgment}
The author is grateful to Professor T.Ohsawa for his patient direction and constant support for about three years, and to the referee for his valuable comments.


\begin{thebibliography}{9}
\bibitem{Hor}
L.H\"{o}rmander, Linear partial differential operators, Springer-Verlag, Berlin Heidelberg New York, 1963
\bibitem{Jac}
H.Jacobowitz, \textit{The canonical bundle and realizable CR hypersurfaces}, Pacific J. of Math., \textrm{\boldmath $127$}, No.$1$ $(1987)$
\bibitem{JaTr}
H.Jacobowitz and F.Treves, \textit{Non-realizable CR structures}, Invent. Math., \textrm{\boldmath $66$}$(1982)$, $231-249$
\bibitem{Leb}
C.LeBrun, \textit{Twistor CR manifolds and three-dimentional conformal geometry}, Trans.Amer.Math.Soc., \textrm{\boldmath $284$}$(1984)$, $601-616$
\bibitem{Lee}
J.M.Lee, \textit{Pseudo-Einstein structures on CR manifolds}, American J. Math., \textrm{\boldmath $110$}$(1988)$, $157-178$
\bibitem{Web}
S.Webster, \textit{The integrability problem for CR vector bundles}, Proc Symp in Pure Math., \textrm{\boldmath $52$}$(1991)$, Part $3$ 
\end{thebibliography}
\end{document}